\documentclass{article}
\usepackage{amssymb}

\newtheorem{theorem}{Theorem}[section]

\newtheorem{corollary}[theorem]{Corollary}

\newtheorem{definition}[theorem]{Definition}

\newtheorem{lemma}[theorem]{Lemma}

\newenvironment{proof}[1][Proof]{\noindent\textbf{#1.} }{\ \rule{0.5em}{0.5em}}
\input{tcilatex}
\begin{document}

\title{The group of quasi-isometries of the real line cannot act effectively
on the line}
\author{Shengkui Ye, Yanxin Zhao}
\maketitle

\begin{abstract}
We prove that the group $\mathrm{QI}^{+}(\mathbb{R})$ of
orientation-preserving quasi-isometries of the real line is a
left-orderable, non-simple group, which cannot act effectively on the real
line $\mathbb{R}.$
\end{abstract}

\section{Introduction}

A function $f:X\rightarrow Y$ between metric spaces $X,Y$ is a
quasi-isometry, if there exist real numbers $K\geq 1$, $C\geq 0$ such that 
\[
\frac{1}{K}d(x_{1},x_{2})-C\leq d(f(x_{1}),f(x_{2}))\leq Kd(x_{1},x_{2})+C, 
\]%
for any $x_{1},x_{2}\in X$ and $d(\func{Im}f,y)\leq C$ for any $y\in Y.$ Two
quasi-isometries $f,g$ are called equivalent if they are of bounded distance 
$\sup_{x\in X}d(f(x),g(x))<\infty $. The quasi-isometry group $\mathrm{QI}%
(X) $ is the group of all equivalence classes $[f]$ of quasi-isometries $%
f:X\rightarrow X$ under composition. The notion of quasi-isometries is one
of the fundamental concepts in geometric group theory. In this note, we
consider the quasi-isometry group $\mathrm{QI}(\mathbb{R})$ of the real
line. Gromov and Pansu \cite{gp} (Section 3.3B) noted that the group of
bi-Lipschitz homeomorphisms has a full image in $\mathrm{QI}(\mathbb{R})$.
Sankaran \cite{sar} proved that the orientation-preserving subgroup $\mathrm{%
QI}^{+}(\mathbb{R})$ is torsion-free and many large groups, like Thompson
groups, free groups of infinite rank, can be embedded into $\mathrm{QI}^{+}(%
\mathbb{R}).$

Recall that a group $G$ is left-orderable if there is a total order $\leq $
on $G$ such that $g\leq h$ implies $fg\leq fh$ for any $f\in G.$ We will
prove the following.

\begin{theorem}
\label{th0}The quasi-isometry group $\mathrm{QI}^{+}(\mathbb{R})$ (or $%
\mathrm{QI}([0,+\infty ))$) is not simple.
\end{theorem}

\begin{theorem}
\label{th1}The quasi-isometry group $\mathrm{QI}^{+}(\mathbb{R})$ (or $%
\mathrm{QI}([0,+\infty ))$) is left-orderable.
\end{theorem}

\begin{theorem}
\label{th2}The quasi-isometry group $\mathrm{QI}^{+}(\mathbb{R})$ cannot act
effectively on the real line $\mathbb{R}.$
\end{theorem}

Other (uncountable) left-orderable groups, but cannot act on the line have
been known to Mann and C. Rivas \cite{mann} (the germ group $\mathcal{G}%
_{\infty }(\mathbb{R})$), Chen-Mann \cite{cm} (the compact supported
diffeomorphism group $\mathrm{Diff}_{c}(\mathbb{R}^{n}),n>1$), and so on.

\section{The group structure of $\mathrm{QI}(\mathbb{R})$}

Let $\mathrm{QI}(\mathbb{R}_{+})$ (resp. $\mathrm{QI}(\mathbb{R}_{-})$) be
the quasi-isometry group of the ray $[0,+\infty )$ (resp. $(-\infty ,0]$),
viewed as subgroup of $\mathrm{QI}(\mathbb{R})$ fixing the negative part.

\begin{lemma}
\label{lem1}$\mathrm{QI}(\mathbb{R})=(\mathrm{QI}(\mathbb{R}_{+})\times 
\mathrm{QI}(\mathbb{R}_{-}))\rtimes \langle t\rangle ,$ where $t\in \mathrm{%
QI}(\mathbb{R})$ is the reflection $t(x)=-x$ for any $x\in \mathbb{R}.$
\end{lemma}

\begin{proof}
Sankaran \cite{sar} proves that the group $PL_{\delta }(\mathbb{R})$
consisting of piecewise linear homeomorphisms with bounded slopes has a full
image in $\mathrm{QI}(\mathbb{R}).$ Since every homeomorphism $f\in $ $%
PL_{\delta }(\mathbb{R})$ is of bounded distance to the map $f-f(0)\in
PL_{\delta }(\mathbb{R}),$ we see that the subgroup 
\[
PL_{\delta ,0}(\mathbb{R})=\{f\in PL_{\delta }(\mathbb{R})\mid f(0)=0\} 
\]%
also has full image in $\mathrm{QI}(\mathbb{R}).$ Let%
\begin{eqnarray*}
PL_{\delta ,+}(\mathbb{R}) &=&\{f\in PL_{\delta }(\mathbb{R})\mid
f(x)=x,x\leq 0\}, \\
PL_{\delta ,-}(\mathbb{R}) &=&\{f\in PL_{\delta }(\mathbb{R})\mid
f(x)=x,x\geq 0\}.
\end{eqnarray*}%
Since $PL_{\delta ,+}(\mathbb{R})\cap PL_{\delta ,-}(\mathbb{R})=\{\mathrm{id%
}_{\mathbb{R}}\},$ we see that $PL_{\delta ,+}(\mathbb{R})\times PL_{\delta
,-}(\mathbb{R})$ has a full image in $\mathrm{QI}^{+}(\mathbb{R}),$ the
orientation-preserving subgroup of $\mathrm{QI}(\mathbb{R}).$ It's obvious
that $PL_{\delta ,+}(\mathbb{R})$ (resp. $PL_{\delta ,-}(\mathbb{R})$) has a
full image in $\mathrm{QI}(\mathbb{R}_{+})$ (resp. $\mathrm{QI}(\mathbb{R}%
_{-})$). Therefore, $\mathrm{QI}(\mathbb{R})=(\mathrm{QI}(\mathbb{R}%
_{+})\times \mathrm{QI}(\mathbb{R}_{-}))\rtimes \langle t\rangle $ $.$
\end{proof}

\bigskip

Let $\mathrm{Homeo}_{+}(\mathbb{R})$ be the group of orientation-preserving
homeomorphisms of the real line. Two functions $f,g\in \mathrm{Homeo}_{+}(%
\mathbb{R})$ are of bounded distance if $\sup_{|x|\geq M}|f(x)-g(x)|<\infty $
for sufficiently large real number $M.$ This means when we study elements $%
[f]$ in $\mathrm{QI}(\mathbb{R}),$ we don't need to care too much about the
function values $f(x)$ for $x$ with small absolute values. We will
implicitly use this fact in the following context. As $PL_{\delta }(\mathbb{R%
})$ has a full image in $\mathrm{QI}(\mathbb{R})$ (by Sankaran \cite{sar}),
we take representatives of quasi-isometries which are homeomorphisms in the
rest of the article.

\subsection{$\mathrm{QI}(\mathbb{R}_{+})$ is not simple}

Let $\mathrm{QI}(\mathbb{R}_{+})$ be the quasi-isometry group of the
half-line $[0,+\infty )$. Note that the quasi-isometry group $\mathrm{QI}%
^{+}(\mathbb{R})$ $=\mathrm{QI}(\mathbb{R}_{+})\times \mathrm{QI}(\mathbb{R}%
_{-})$ and $\mathrm{QI}(\mathbb{R}_{+})\cong \mathrm{QI}(\mathbb{R}_{-}),$
by Lemma \ref{lem1}. Let $H=\{[f]\in \mathrm{QI}(\mathbb{R}_{+})\mid f$ is
quasi-isometric such that $\lim_{x\rightarrow \infty }\frac{f(x)-x}{x}=0\}.$
Theorem \ref{th0} follows from the following theorem.

\begin{theorem}
The $H$ is a proper normal subgroup of $\mathrm{QI}(\mathbb{R}_{+}).$ In
particular, $\mathrm{QI}(\mathbb{R}_{+})$ is not simple.
\end{theorem}

\begin{proof}
For any $[f],[g]\in H,$ we have%
\[
\frac{f(g(x))-x}{x}=\frac{f(g(x))-g(x)}{g(x)}\frac{g(x)}{x}+\frac{g(x)-x}{x}%
. 
\]%
Since $g$ is a quasi-isometry, we know that $\frac{1}{K}x-C\leq
g(x)-g(0)\leq Kx+C.$ Therefore, $\frac{1}{K}-1\leq \frac{g(x)}{x}\leq K+1$
for sufficiently large $x.$ When $x\rightarrow \infty ,$ we have $%
g(x)\rightarrow \infty .$ This means $\frac{f(g(x))-g(x)}{g(x)}\rightarrow
0. $ Therefore, $\frac{f(g(x))-x}{x}\rightarrow 0$ as $x\rightarrow \infty .$
This proves that $[fg]\in H.$

Note that 
\[
\frac{|f^{-1}(x)-x|}{x}=\frac{|f^{-1}(x)-f^{-1}(f(x))|}{x}\leq \frac{%
K|x-f(x)|+C}{x}. 
\]%
Therefore, 
\[
\lim_{x\rightarrow \infty }\frac{|f^{-1}(x)-x|}{x}=0. 
\]%
This means $[f^{-1}]\in H$ and that $H$ is a subgroup.

For any quasi-isometric homeomorphism $g\in \mathrm{Homeo}(\mathbb{R}_{+})$
and any $[f]\in H,$ we have%
\begin{eqnarray*}
\frac{g^{-1}(f(g(x)))-x}{x} &=&\frac{g^{-1}(f(g(x)))-g^{-1}(g(x))}{x} \\
&=&\frac{g^{-1}(f(g(x)))-g^{-1}(g(x))}{g(x)}\frac{g(x)}{x}.
\end{eqnarray*}%
Note that when $x\rightarrow \infty ,$ the function $\frac{g(x)}{x}$ is
bounded. Let $y=g(x).$ We have%
\[
\frac{|g^{-1}(f(y))-g^{-1}(y)|}{y}\leq \frac{K|f(y)-y|+C}{y}\rightarrow
0,x\rightarrow \infty . 
\]%
Therefore, $[g^{-1}fg]\in H.$

It's obvious that the function $f$ defined by $f(x)=2x$ is not an element in 
$H.$ The function defined by $g(x)=x+\ln (x+1)$ gives a non-trivial element
in $H.$ This means that $H$ is a proper normal subgroup of $\mathrm{QI}(%
\mathbb{R}_{+}).$
\end{proof}

\bigskip

\begin{lemma}
\label{wcor}Let 
\[
W(\mathbb{R})=\{f\in \mathrm{Diff}(\mathbb{R})\mid \sup_{x\in \mathbb{R}%
}|f(x)-x|<+\infty ,\sup_{x\in R}|f^{\prime }(x)|<\infty \} 
\]%
be the group consisting of diffeomorphisms with bounded derivatives and of
bounded distance from the identity. Define a homeomorphism $h:\mathbb{R}%
\rightarrow \mathbb{R}$ by $h(x)=e^{x}$ when $x\geq 1,$ $h(x)=-h(-x)$ when $%
x\leq -1,$ and $h(x)=ex$ when $-1\leq x\leq 1.$ Then $hfh^{-1}\ $is a
quasi-isometry for any $f\in W(\mathbb{R})$.
\end{lemma}

\begin{proof}
For any $f\in W(\mathbb{R})$ and sufficiently large $x>0,$ its derivative
satisfies that 
\begin{eqnarray*}
|hfh^{-1}(x)^{\prime }| &=&|(e^{f(\ln x)})^{\prime }|=|(xe^{f(\ln x)-\ln
x})^{\prime }| \\
&=&|e^{f(\ln x)-\ln x}(1+f^{\prime }(\ln x)-1)|=|e^{f(\ln x)-\ln x}f^{\prime
}(\ln x)| \\
&\leq &e^{\sup_{x\in \mathbb{R}}|f(x)-x|}\cdot \sup_{x\in \mathbb{R}%
}|f^{\prime }(x)|.
\end{eqnarray*}%
The case for negative $x<0$ can be calculated similarly. This proves that $%
hfh^{-1}\ $is a quasi-isometry.
\end{proof}

\bigskip

The following result was proved by Sankaran \cite{sar}.

\begin{corollary}
\label{corsar}The quasi-isometry group $\mathrm{QI}(\mathbb{R})$ contains $%
\mathrm{Diff}_{\mathbb{Z}}(\mathbb{R})$ (lifting of $\mathrm{Diff}(S^{1})$
to $\mathrm{Homeo}(\mathbb{R})$).
\end{corollary}

\begin{proof}
For any $f\in \mathrm{Diff}_{\mathbb{Z}}(\mathbb{R}),$ we have $%
f(x+1)=f(x)+1 $ for any $x\in \mathbb{R}.$ This means $\sup_{x\in \mathbb{R}%
}|f(x)-x|<+\infty .$ Since $f(x)-x$ is periodic, we know that $f^{\prime
}(x) $ is bounded. Suppose that $f(x)>x$ for some $x\in \lbrack 0,1].$ Take $%
y_{n}=e^{x+n},n>0.$ Let $h$ be the function defined in Lemma \ref{wcor}. We
have%
\[
|hfh^{-1}(y_{n})-y_{n}|=|e^{f(x+n)}-e^{x+n}|=|e^{f(x)}-e^{x}|e^{n}%
\rightarrow \infty , 
\]%
which means $[hfh^{-1}]\neq \lbrack id]\in \mathrm{QI}(\mathbb{R}).$
\end{proof}

\begin{lemma}
The quasi-isometry group $\mathrm{QI}(\mathbb{R})$ contains the semi-direct
product $\mathrm{Diff}_{\mathbb{Z}}(\mathbb{R})\ltimes H.$
\end{lemma}

\begin{proof}
Since $H\ $is normal, it's enough to prove that $\mathrm{Diff}_{\mathbb{Z}}(%
\mathbb{R})\cap H=\{e\},$ the trivial subgroup. Actually, for any $f\in 
\mathrm{Diff}_{\mathbb{Z}}(\mathbb{R}),$ the conjugate $hfh^{-1}$ is a
quasi-isometry as in the proof of Corollary \ref{corsar}. If $hfh^{-1}\in H,$
we have that 
\[
\lim_{x\rightarrow \infty }\frac{hfh^{-1}(x)}{x}=\lim_{x\rightarrow \infty }%
\frac{xe^{f(\ln x)-\ln x}}{x}=\lim_{x\rightarrow \infty }e^{f(\ln x)-\ln
x}=1. 
\]%
Since $f(x)-x$ is periodic, we know that $f(\ln x)=\ln x$ for any
sufficiently large $x.$ But this means that $f(y)=y$ for any $y$ and $f$ is
the identity.
\end{proof}

\subsection{Affine subgroups of $\mathrm{QI}(\mathbb{R})$}

\begin{lemma}
\label{lema1}The quasi-isometry group $\mathrm{QI}(\mathbb{R}_{+})$
(actually the semi-direct product $\mathrm{Diff}_{\mathbb{Z}}(\mathbb{R}%
)\ltimes H$) contains the semi-direct product $\mathbb{R}_{>0}\ltimes
(\bigoplus_{i\in \mathbb{R}_{\geq 1}}\mathbb{R)},$ generated by $%
A_{t},B_{i,s},$ $t\in \mathbb{R}_{>0},i\in \mathbb{R}_{\geq 1}=[1,\infty
),s\in \mathbb{R}$ satisfying%
\begin{eqnarray*}
A_{t}B_{i,s}A_{t}^{-1} &=&B_{i,st^{\frac{i}{i+1}}}, \\
B_{i,s_{1}}B_{i,s_{2}} &=&B_{i,s_{1}+s_{2}}, \\
B_{i,s_{1}}B_{j,s_{2}} &=&B_{j,s_{2}}B_{i,s_{1}} \\
A_{t_{1}}A_{t_{2}} &=&A_{t_{1}t_{2}}
\end{eqnarray*}%
for any $t_{1},t_{2}\in \mathbb{R}_{>0},i,j\in \mathbb{R}_{\geq
1},s_{1},s_{2}\in \mathbb{R}$.
\end{lemma}

\begin{proof}
Let 
\begin{eqnarray*}
A_{t}(x) &=&tx,t\in \mathbb{R}_{>0}, \\
B_{i,s}(x) &=&x+sx^{\frac{1}{i+1}},s\in \mathbb{R}
\end{eqnarray*}%
for $x\geq 0.$ We define $A_{t}(x)=B_{i,s}(x)=x$ for $x\leq 0.$ Since the
derivatives 
\[
A_{t}^{\prime }(x)=t,B_{i,s}^{\prime }(x)=1+\frac{s}{i+1}x^{\frac{-i}{i+1}} 
\]%
are bounded for sufficiently large $x,$ we know that $A_{t},B_{i,s}$ are
quasi-isometries. For any $x\geq 1,$ we have 
\begin{eqnarray*}
A_{t}B_{i,s}A_{t}^{-1}(x) &=&A_{t}B_{i,s}(x/t)=A_{t}(x/t+s(x/t)^{\frac{1}{i+1%
}}) \\
&=&x+st^{\frac{i}{i+1}}x^{\frac{1}{i+1}}=B_{i,st^{\frac{i}{i+1}}}(x).
\end{eqnarray*}

For any $x\geq 1,$ we have $B_{i,s_{1}}B_{i,s_{2}}(x)=B_{i,s_{1}}(x+s_{2}x^{%
\frac{1}{i+1}})=x+s_{2}x^{\frac{1}{i+1}}+s_{1}(x+s_{2}x^{\frac{1}{i+1}})^{%
\frac{1}{i+1}}$ and 
\begin{eqnarray*}
|B_{i,s_{1}}B_{s_{2}}(x)-B_{i,s_{1}+s_{2}}(x)| &=&|s_{1}((x+s_{2}x^{\frac{1}{%
i+1}})^{\frac{1}{i+1}}-x^{\frac{1}{i+1}})| \\
&\leq &|s_{1}\frac{s_{2}x^{\frac{1}{i+1}}}{x^{\frac{i}{i+1}}}|\leq
|s_{1}s_{2}|,
\end{eqnarray*}%
by Newton's Binomial Theorem. This means that $B_{i,s_{1}}B_{i,s_{2}}$ and $%
B_{i,s_{1}+s_{2}}$ are of bounded distance. It is obvious that $%
A_{t_{1}}A_{t_{2}}=A_{t_{1}t_{2}}.$

When $i<j$ are distinct natural numbers, 
\begin{eqnarray*}
&&|B_{i,s_{1}}B_{j,s_{2}}(x)-B_{j,s_{2}}B_{i,s_{1}}(x)| \\
&=&|x+s_{2}x^{\frac{1}{j+1}}+s_{1}(x+s_{2}x^{\frac{1}{j+1}})^{\frac{1}{i+1}%
}-(x+s_{1}x^{\frac{1}{i+1}}+s_{2}(x+s_{1}x^{\frac{1}{i+1}})^{\frac{1}{j+1}})|
\\
&=&|s_{1}((x+s_{2}x^{\frac{1}{j+1}})^{\frac{1}{i+1}}-x^{\frac{1}{i+1}%
})+s_{2}(x^{\frac{1}{j+1}}-(x+s_{1}x^{\frac{1}{i+1}})^{\frac{1}{j+1}})| \\
&\leq &|s_{1}\frac{s_{2}x^{\frac{1}{j+1}}}{x^{\frac{i}{i+1}}}|+|s_{2}\frac{%
s_{1}x^{\frac{1}{i+1}}}{x^{\frac{j}{j+1}}}|\leq 2|s_{1}s_{2}|
\end{eqnarray*}%
for any $x\geq 1.$ This proves that images $[A_{t}],[B_{i,s}]\in \mathrm{QI}(%
\mathbb{R}_{\geq 0})$ satisfy the relations. By abuse of notations, we still
denote the classes by the same letters.

We prove that the subgroup generated by $\{B_{i,s},$ $i\in \mathbb{R}_{\geq
1},s\in \mathbb{R\}}$ is the infinite direct sum $\bigoplus_{i\in \mathbb{R}%
_{\geq 1}}\mathbb{R}$. It's enough to prove that $%
B_{i_{1},s_{1}},B_{i_{2},s_{2}},\cdots ,B_{i_{k},s_{k}}$ are $\mathbb{Z}$%
-linearly independent for distinct $i_{1},i_{2},\cdots ,i_{k}$ and nonzero $%
s_{1},s_{2},...,s_{k}\in \mathbb{R}.$ This can directly checked. For
integers $n_{1},n_{2},...,n_{k},$ suppose that the composite $%
B_{i_{1},s_{1}}^{n_{1}}\circ B_{i_{2},s_{2}}^{n_{2}}\circ \cdots \circ
B_{i_{k},s_{k}}^{n_{k}}=\mathrm{id}\in \mathrm{QI}(\mathbb{R}_{\geq 0}).$ We
have 
\begin{eqnarray*}
&&\sup_{x\in \mathbb{R}_{>0}}|B_{i_{1},s_{1}}^{n_{1}}\circ
B_{i_{2},s_{2}}^{n_{2}}\circ \cdots \circ B_{i_{k},s_{k}}^{n_{k}}(x)-x| \\
&=&\sup_{x\in \mathbb{R}_{>0}}|n_{k}s_{k}x^{\frac{1}{i_{k}+1}%
}+n_{k-1}s_{k-1}(x+n_{k}s_{k}x^{\frac{1}{i_{k}+1}})^{\frac{1}{i_{k-1}+1}%
}+\cdots +n_{1}s_{1}(x+\cdots )^{\frac{1}{i_{1}+1}}|<+\infty ,
\end{eqnarray*}%
which implies $n_{1}=n_{2}=\cdots =n_{k}=0$ considering the exponents.

The subgroup $\mathbb{R}_{>0}\ltimes (\bigoplus_{i\in \mathbb{R}_{\geq 1}}%
\mathbb{R)}$ lies in $\mathrm{Diff}_{\mathbb{Z}}(\mathbb{R})\ltimes H$ by
the following construction. Let $a_{t},b_{i,s}:\mathbb{R}\rightarrow \mathbb{%
R}\ $be defined by $a_{t}(x)=x+\ln t,b_{i,s}(x)=\ln (e^{x}+se^{\frac{x}{i+1}%
})$ for $t\in \mathbb{R}_{>0},i\in \mathbb{R}_{\geq 1},s\in \mathbb{R}$. It
can be directly checked that $a_{t}\in $ $\mathrm{Diff}_{\mathbb{Z}}(\mathbb{%
R})$ and $b_{i,s}\in W(\mathbb{R})$ (defined in Lemma \ref{wcor}). Let $%
h(x)=e^{x}.$ A direct calculation shows that $%
ha_{t}h^{-1}=A_{t},hb_{i,s}h^{-1}=B_{i,s},$ as elements in $\mathrm{QI}(%
\mathbb{R}_{+}).$
\end{proof}

\section{Left-orderability}

The following is well-known (for a proof, see \cite{nav}, Prop. 1.4.)

\begin{lemma}
\label{3.1}A group $G$ is left-orderable if and only if, for every finite
collection of nontrivial elements $g_{1},...,g_{k}$, there exist choices $%
\varepsilon _{i}\in \{1,-1\}$ such that the identity is not an element of
the semigroup generated by $\{g_{i}^{\varepsilon _{i}},i=1,2,...,k\}.$
\end{lemma}

The proof of Theorem \ref{th1} follows a similar strategy used by Navas to
prove the left-orderability of the group $\mathcal{G}_{\infty }$ of germs at 
$\infty $ of homeomorphisms of $\mathbb{R}$ (cf. \cite{dnr}, Remark 1.1.13, 
\cite{mann}, Prop. 2.2).

\begin{proof}[Proof of Theorem \protect\ref{th1}]
It's enough to prove that $\mathrm{QI}(\mathbb{R}_{+})$ is left-orderable.
Let $f_{1},f_{2},...,f_{n}\in \mathrm{QI}(\mathbb{R}_{+})$ be any finitely
many non-trivial elements. Note that any $1\neq \lbrack f]\in \mathrm{QI}(%
\mathbb{R}_{+})$ has $\sup_{x>0}|f(x)-x|=\infty .$ This property doesn't
depend on the choice of $f\in \lbrack f].$ Without confusions, we still
denote $[f]$ by $f.$ Choose a sequence $\{x_{1,k}\}\subset \mathbb{R}_{+}$
such that $\sup_{k\in \mathbb{N}}|f_{1}(x_{1,k})-x_{1,k}|=\infty .$ For each 
$i>1,$ we have either $\sup_{k\in \mathbb{N}}|f_{i}(x_{1,k})-x_{1,k}|=\infty 
$ or $\sup_{k\in \mathbb{N}}|f_{i}(x_{1,k})-x_{1,k}|\leq M$ for a real
number $M.$ After passing to subsequences, we assume for each $i=1,2,...,n$
that $f_{i}(x_{1,k})-x_{1,k}\rightarrow +\infty ,$ or $%
f_{i}(x_{1,k})-x_{1,k}\rightarrow -\infty $ or $\sup_{k\in \mathbb{N}%
}|f_{i}(x_{1,k})-x_{1,k}|\leq M.$ We assign $\varepsilon _{i}=1$ for the
first case, $\varepsilon _{i}=-1$ for the second case. For the third case,
let 
\[
S_{1}=\{f_{i}\mid \sup_{k\in \mathbb{N}}|f_{i}(x_{1,k})-x_{1,k}|\leq M\}. 
\]%
Note that $f_{1}\notin S_{1}.$ Choose $f_{i_{0}}\in S_{1}$ if $S_{1}$ is not
empty. We choose another sequence $\{x_{2,k}\}$ such that $\sup_{k\in 
\mathbb{N}}|f_{i_{0}}(x_{2,k})-x_{2,k}|=\infty .$ Similarly, after passing
to a subsequence, we have for each $f\in S_{1}$ that $f(x_{2,k})-x_{2,k}%
\rightarrow +\infty ,$ or $f(x_{2,k})-x_{2,k}\rightarrow -\infty $ or $%
\sup_{k\in \mathbb{N}}|f(x_{2,k})-x_{2,k}|\leq M^{\prime }$ for another real
number $M^{\prime }.$ Assign $\varepsilon _{i}=1$ for the first case and $%
\varepsilon _{i}=-1$ for the second case. Continue this process to define $%
S_{2},S_{3},...$ and choose sequences $\{x_{i,k}\},i=3,4,...,$ to assign $%
\varepsilon _{i}$ for each $f_{i}.$ Note that the process will stop at $n$
times, as the number of elements without assignment is strictly decreasing.

For an element $f\in \mathrm{QI}(\mathbb{R}_{+})$ satisfying $%
f(x_{i})-x_{i}\rightarrow \infty ,i\rightarrow \infty $ for some sequence $%
\{x_{i}\},$ we assume that $f(x_{i})-x_{i}>0$ for each $i.$ Since $f,f^{-1}$
are orientation-preserving, we have 
\begin{eqnarray*}
f^{-1}(x_{i})-x_{i}
&=&-(x_{i}-f^{-1}(x_{i}))=-(f^{-1}(f(x_{i}))-f^{-1}(x_{i})) \\
&\leq &-(\frac{1}{K}(f(x_{i})-x_{i})-C)\rightarrow -\infty .
\end{eqnarray*}

Let $w=f_{i_{1}}^{\varepsilon _{i_{1}}}\cdots f_{i_{m}}^{\varepsilon
_{i_{m}}}\in \langle f_{1},f_{2},...,f_{n}\rangle $ be a nontrivial word. If 
$\{i_{1},...,i_{m}\}\nsubseteq S_{1},$ we have $w(x_{1,k})-x_{1,k}%
\rightarrow \infty .$ Otherwise, $\sup_{k\in \mathbb{N}%
}|w(x_{1,k})-x_{1,k}|<\infty .$ Suppose that $\{i_{1},...,i_{m}\}\subset
S_{t},$ but $\{i_{1},...,i_{m}\}\nsubseteq S_{t+1}$ with the assumption that 
$S_{0}=\{f_{1},f_{2},...,f_{n}\}.$ We have $w(x_{t+1,k})-x_{t+1,k}%
\rightarrow \infty ,k\rightarrow \infty .$ This proves that $w\neq 1\in 
\mathrm{QI}(\mathbb{R}_{+}).$ Therefore, $\mathrm{QI}(\mathbb{R}_{+})$ is
left-orderable by Lemma \ref{3.1}.
\end{proof}

\begin{lemma}
The group $\mathrm{QI}(\mathbb{R}_{+})$ is not locally indicable.
\end{lemma}

\begin{proof}
Note that $\mathrm{QI}(\mathbb{R}_{+})$ contains the lifting $\tilde{\Gamma}$
of $\mathrm{PSL}(2,\mathbb{R})<\mathrm{Diff}(S^{1})$ to $\mathrm{Homeo}(%
\mathbb{R})$ (cf. Corollary \ref{corsar}). But this lifting $\tilde{\Gamma}$
contains a subgroup $\Gamma =\langle f,g,h:f^{2}=g^{3}=h^{7}=fgh\rangle ,$
the lifting of the $(2,3,7)$-triangle group. There are no non-trivial maps
from $\Gamma $ to $(\mathbb{R},+)$ (for more details see \cite{dnr}, page
94).
\end{proof}

\section{The quasi-isometric group cannot act effectively on the line}

The following was proved by Mann \cite{mann} (Proposition 6).

\begin{lemma}
\label{lema2}Let the affine group $\mathbb{R}_{>0}\ltimes \mathbb{R},$
generated by $A_{t},B_{s},$ $t\in \mathbb{R}_{>0},s\in \mathbb{R}$ satisfying%
\begin{eqnarray*}
A_{t}B_{s}A_{t}^{-1} &=&B_{ts}, \\
B_{s_{1}}B_{s_{2}} &=&B_{s_{1}+s_{2}}, \\
A_{t_{1}}A_{t_{2}} &=&A_{t_{1}t_{2}}.
\end{eqnarray*}%
The affine group $\mathbb{R}_{>0}\ltimes \mathbb{R}$ cannot act effectively
on the real line $\mathbb{R}$ by homeomorphisms with $A_{t}$ a translation
for each $t$.
\end{lemma}

\begin{proof}
Suppose that $\mathbb{R}_{>0}\ltimes \mathbb{R}$ act effectively on the real
line $\mathbb{R}$ with each $A_{t}$ a translation. After passing to index-2
subgroup, we assume that the group is orientation-preserving. If $B_{1}$
acts freely on $\mathbb{R},$ then it is conjugate to the translation $T:%
\mathbb{R}\rightarrow \mathbb{R},x\rightarrow x+1.$ In such a case, we have $%
A_{2}TA_{2}^{-1}=T^{2}.$ Therefore, $A_{2}^{-1}(x+2)=A_{2}^{-1}(x)+1$ for
any $x.$ Since $A_{2}^{-1}$ maps intervals of length $2$ to an interval of
length $1,$ it is contracting and thus has a fixed point.

If $B_{1}$ has a non-empty fixed point set $\mathrm{Fix}(B_{1}),$ choose $I$
to be a connected component of $\mathbb{R}\backslash \mathrm{Fix}(B_{1}).$
Suppose that $A_{2}(x)=x+a,$ a translation by some real number $a>0.$ Since $%
A_{2}=A_{2^{1/n}}^{n},$ we have $A_{2^{1/n}}(x)=x+a/n$ for each positive
integer $n.$ For each $n,$ let $F_{n}=A_{2^{1/n}}B_{1}A_{2^{1/n}}^{-1}.$
Since $A_{2^{1/n}}B_{1}A_{2^{1/n}}^{-1}$ commutes with $B_{1},$ we see that $%
F_{n}\mathrm{Fix}(B_{1})=\mathrm{Fix}(B_{1}).$ This means that either $%
F_{n}(I)=I$ or $F_{n}(I)\cap I=\emptyset .$ Since $F_{n}(x)=B_{1}(x-a/n)+a/n$
for any $x\in \mathbb{R},$ we know that $F_{n}(I)=I$ for sufficiently large $%
n.$ Without loss of generality, we assume that $I$ is of the form $(x,y)$ or 
$(-\infty ,y).$ Choose sufficiently large $n$ such that $y-a/n\in I.$ We
have 
\[
A_{2^{1/n}}B_{1}A_{2^{1/n}}^{-1}(y)=B_{1}(y-a/n)+a/n\neq y, 
\]%
which is a contradiction to the fact that $F_{n}(I)=I.$
\end{proof}

\begin{definition}
A topologically diagonal embedding of a group $G<\mathrm{Homeo}(\mathbb{R})$
is a homomorphism $\phi :G\rightarrow \mathrm{Homeo}_{+}(\mathbb{R})$
defined as follows. Choose a collection of disjoint open intervals $%
I_{n}\subset \mathbb{R}$ and homeomorphisms $f_{n}:\mathbb{R\rightarrow }%
I_{n}$. Define $\phi $ by $\phi (g)(x)=f_{n}gf_{n}^{-1}(x)$ when $x\in I_{n}$
and $\phi (g)(x)=x$ when $x\notin I_{n}.$
\end{definition}

The following is similar to a result proved by Militon \cite{mil}.

\begin{lemma}
\label{mili}(Militon \cite{mil}) Let $\Gamma =\mathrm{PSL}_{2}(\mathbb{R})$
and $\tilde{\Gamma}<\mathrm{Homeo}_{+}(\mathbb{R})$ the lifting of $\Gamma $
to the real line. Any effective action $\phi :\tilde{\Gamma}\hookrightarrow 
\mathrm{Homeo}_{+}(\mathbb{R})$ of $\tilde{\Gamma}$ on the real line $%
\mathbb{R}$ is a topological diagonal embedding$.$
\end{lemma}

\begin{proof}
After passing to index-2 subgroup, we assume the action is
orientation-preserving. Let $\tau :\mathbb{R}\rightarrow \mathbb{R}$ be the
translation $x\rightarrow x+1.$ Suppose that $\mathrm{Fix}(\phi (\tau ))\neq
\emptyset .$ Note that $\tau $ lies in the center of $\tilde{\Gamma}.$ The
quotient group $\Gamma =\tilde{\Gamma}/\langle \tau \rangle $ acts on the
fixed point set $\mathrm{Fix}(\phi (\tau )).$ For any $f\in \Gamma $ and $%
x\in \mathrm{Fix}(\phi (\tau )),$ we denote the action by $f(x)$ without
confusions. Choose any torsion-element $f\in \Gamma $ and any $x\in \mathrm{%
Fix}(\phi (\tau )).$ We must have $x=f(x),$ for otherwise $%
x<f(x)<f^{2}(x)<\cdots <f^{k}(x)$ for any $k.$ Since $\Gamma $ is simple, we
know that the action of $\tilde{\Gamma}$ on $\mathrm{Fix}(\tau )$ is
trivial. For each connected component $I_{i}\subset \mathbb{R}\backslash 
\mathrm{Fix}(\phi (\tau )),$ we know that $\tau |_{I_{i}}$ is conjugate to a
translation. The group $\Gamma =\tilde{\Gamma}/\langle \tau \rangle $ acts
on $I_{i}/\langle \phi (\tau )\rangle =S^{1}.$ A result of Matsumoto \cite%
{ma} (Theorem 5.2) says that the group $\Gamma $ is conjugate to the natural
inclusion $\mathrm{PSL}_{2}(\mathbb{R})\hookrightarrow \mathrm{Homeo}%
_{+}(S^{1})$ by a homeomorphism $g\in \mathrm{Homeo}_{+}(S^{1}).$ Therefore,
the group $\phi (\tilde{\Gamma})|_{I_{i}}$ is conjugate to the image of the
natural inclusion $\tilde{\Gamma}\hookrightarrow \mathrm{Homeo}_{+}(\mathbb{R%
}).$
\end{proof}

For a real number $a\in \mathbb{R},$ let 
\begin{eqnarray*}
t_{a} &:&\mathbb{R}\rightarrow \mathbb{R}, \\
x &\rightarrow &x+a
\end{eqnarray*}%
be the translation. Denote by $A=\langle t_{a}:a\in \mathbb{R}\rangle ,$ the
subgroup of translations in the lifting $\tilde{\Gamma}$ of $\mathrm{PSL}%
_{2}(\mathbb{R}).$

\begin{corollary}
\label{continu}For any injective group homomorphism $\phi :\tilde{\Gamma}%
\rightarrow \mathrm{Homeo}(\mathbb{R})$, the image $\phi (A)$ is a
continuous one-parameter subgroup, i.e. $\lim_{a\rightarrow a_{0}}\phi
(t_{a})=\phi (t_{a_{0}})$ for any $a_{0}\in \mathbb{R}.$
\end{corollary}

\begin{proof}
If $\phi $ is injective, the previous lemma says that $\phi $ is a
topological diagonal embedding. Therefore, $\phi (A)$ is continuous.
\end{proof}

\bigskip

We will need the following elementary fact.

\begin{lemma}
\label{cont}Let $\phi :(\mathbb{R},+)\rightarrow (\mathbb{R},+)$ be a group
homomorphism. If $\phi $ is continuous at any $x\neq 0$, then $\phi $ is $%
\mathbb{R}$-linear.
\end{lemma}

\begin{proof}
For any nonzero integer $n,$ we have $\phi (n)=n\phi (1)$ and $\phi (1)=\phi
(\frac{1}{n}n)=n\phi (\frac{1}{n}).$ Since $\phi $ is additive, we have $%
\phi (\frac{m}{n})=m\phi (\frac{1}{n})=\frac{m}{n}f(1)$ for any integer $%
m,n\neq 0.$ For any nonzero real number $a\in \mathbb{R},$ choose a rational
sequence $r_{i}\rightarrow a.$ When $\phi $ is continuous, we have that $%
\phi (r_{i})\rightarrow \phi (a)$ and $\phi (r_{i})=r_{i}\phi (1)\rightarrow
a\phi (1)=\phi (a).$ The proof is finished.
\end{proof}

The following is the classical theorem of H\"{o}lder: a group acting freely
on $\mathbb{R}$ is semi-conjugate to a group of translations (see Navas \cite%
{na}, Section 2.2.4).

\begin{lemma}
\label{holder}Let $\Gamma $ be a group acting freely on the real line $%
\mathbb{R}.$ There is an injective group homomorphism $\phi :\Gamma
\rightarrow (\mathbb{R},+)$ and a continuous non-decreasing map $\varphi :%
\mathbb{R}\rightarrow \mathbb{R}$ such that%
\[
\varphi (h(x))=\varphi (x)+\phi (h) 
\]%
for any $x\in \mathbb{R},h\in \Gamma .$
\end{lemma}

\begin{corollary}
\label{cor}Suppose that the affine group $\mathbb{R}_{>0}\mathbb{\ltimes R=}%
\langle a_{t}:t\in \mathbb{R}_{>0}\rangle \mathbb{\ltimes }\langle
b_{s}:s\in \mathbb{R}\rangle $ acts on the real line $\mathbb{R}$ by
homeomorphisms satisfying

(1) the action of the subgroup $\mathbb{R=}\langle b_{s}:s\in \mathbb{R}%
\rangle $ is free;

(2) for any fixed $x\in \mathbb{R}$, $a_{t}(x)$ is continuous with respect
to $t\in \mathbb{R}_{>0}.$

Let $\phi :\langle b_{s}:s\in \mathbb{R}\rangle \rightarrow (\mathbb{R},+)$
be the additive map in Lemma \ref{holder} for $\Gamma =\langle b_{s}:s\in 
\mathbb{R}\rangle $. Then $\phi $ is an $\mathbb{R}$-linear map.
\end{corollary}

\begin{proof}
Note that $a_{t}b_{s}a_{t}^{-1}=b_{ts}.$ We have%
\[
\varphi (b_{ts}(x))=\varphi (x)+\phi (b_{ts}). 
\]%
Since $b_{ts}(x)=a_{t}b_{s}a_{t}^{-1}(x)\rightarrow b_{s}(x)$ when $%
t\rightarrow 1,$ we have that%
\[
\varphi (x)+\phi (b_{ts})\rightarrow \varphi (b_{s}(x))=\varphi (x)+\phi
(b_{s}). 
\]%
This implies that $\phi (b_{ts})\rightarrow \phi (b_{s}),t\rightarrow 1.$
For any nonzero $x\in \mathbb{R}$ and sequence $x_{n}\rightarrow x,$%
\[
\phi (b_{x_{n}})=\phi (b_{\frac{x_{n}}{x}x})\rightarrow \phi (b_{x}). 
\]%
The map $\phi $ is $\mathbb{R}$-linear by Lemma \ref{cont}.
\end{proof}

\bigskip

\begin{theorem}
\label{th3}Let $G=\mathbb{R}_{>0}\ltimes (\bigoplus_{i\in \mathbb{R}_{\geq
1}}\mathbb{R)}$, generated by $A_{t},B_{i,s},$ $t\in \mathbb{R}_{>0},i\in 
\mathbb{R}_{\geq 1}=[1,\infty ),s\in \mathbb{R}$ satisfying%
\begin{eqnarray*}
A_{t}B_{i,s}A_{t}^{-1} &=&B_{i,st^{\frac{i}{i+1}}}, \\
B_{i,s_{1}}B_{i,s_{2}} &=&B_{i,s_{1}+s_{2}}, \\
B_{i,s_{1}}B_{j,s_{2}} &=&B_{j,s_{2}}B_{i,s_{1}} \\
A_{t_{1}}A_{t_{2}} &=&A_{t_{1}t_{2}}
\end{eqnarray*}%
for any $t_{1},t_{2}\in \mathbb{R}_{>0},i,j\in \mathbb{R}_{\geq
1},s_{1},s_{2}\in \mathbb{R}$. Then $G$ cannot act effectively on the real
line $\mathbb{R}$ by homeomorphisms when the induced action of $\langle
A_{t}:t\in \mathbb{R}_{>0}\rangle $ is a topologically diagonal embedding of
the translation subgroup $(\mathbb{R},+)\hookrightarrow \mathrm{Homeo}(%
\mathbb{R}).$
\end{theorem}

\begin{proof}
Suppose that $G$ acts effectively on $\mathbb{R}$ with the induced action of 
$\langle A_{t}:t\in \mathbb{R}_{>0}\rangle $ a topologically diagonal
embedding of the translation subgroup $(\mathbb{R},+)\hookrightarrow \mathrm{%
Homeo}(\mathbb{R}).$ Let $I$ be a connected component of $\mathbb{R}%
\backslash \mathrm{Fix}(\langle A_{t},B_{i,s}:t\in \mathbb{R}_{>0},i=1,s\in 
\mathbb{R}\rangle ).$

Suppose that there is an element $B_{1,s}$ having a fixed point $x\in I$ for
some $s>0.$ Since $A_{4}B_{1,s}A_{4}^{-1}=B_{1,s}^{2},$ we know that $%
A_{4}x\in \mathrm{Fix}(B_{1,s})=\mathrm{Fix}(B_{1,s}^{2}).$ Since there are
no fixed points in $I$ for $\langle A_{t},B_{1,s}:t\in \mathbb{R}_{>0},s\in 
\mathbb{R}\rangle ,$ we know that $\lim_{n\rightarrow \infty
}A_{4}^{n}x\notin I$ (otherwise, $\lim_{n\rightarrow \infty }A_{4}^{n}x\in I.
$ But $A_{t}(\lim_{n\rightarrow \infty }A_{4}^{n}x)=\lim_{n\rightarrow
\infty }A_{4}^{n}x$ for any $t>0$ by the topologically diagonal embedding.
For any $s^{\prime },$ we have $B_{1,s^{\prime }}=A_{s^{\prime
2}s^{-2}}B_{1,s}A_{s^{\prime 2}s^{-2}}^{-1}$ and $B_{1,s^{\prime
}}(\lim_{n\rightarrow \infty }A_{4}^{n}x)=\lim_{n\rightarrow \infty
}A_{4}^{n}x.$ This would imply that $\lim_{n\rightarrow \infty }A_{4}^{n}x$
is a global fixed point of $\langle A_{t},B_{1,s}:t\in \mathbb{R}_{>0},s\in 
\mathbb{R}\rangle $). This implies that $A_{4}$ has no fixed point in $I.$
Since the group homomorphism $\langle A_{t}:t\in \mathbb{R}_{>0}\rangle
\rightarrow \mathrm{Homeo}(\mathbb{R})$ is a diagonal embedding, we see that
each $A_{t}$ has no fixed point in $I$ and the action of $\langle A_{t}:t\in 
\mathbb{R}_{>0}\rangle $ on $I$ is conjugate to a group of translations. By
Lemma \ref{lema2}, the affine group $\langle A_{t},B_{1,s}:t\in \mathbb{R}%
_{>0},s\in \mathbb{R}\rangle $ cannot act effectively on $I$. Suppose that $%
A_{t}B_{1,s^{\prime }}$ acts trivially on $I$ for some $t>0$ and $s^{\prime
}>0.$ We have that $A_{t}B_{1,s}=A_{s^{2}s^{\prime -2}}(A_{t}B_{1,s^{\prime
}})A_{s^{2}s^{\prime -2}}^{-1}$ acts trivially on $I.$ But $%
A_{t}B_{1,s}(x)=A_{t}(x)=x$ implies that $t=1.$ Therefore, the element $%
B_{1,s}$ (and any $B_{1,t}=A_{t^{2}s^{-2}}B_{1,s}A_{t^{2}s^{-2}}^{-1},t\in 
\mathbb{R}_{>0}$) acts trivially on $I.$ This means that the action of $%
\langle B_{1,s}:s\in \mathbb{R}\rangle $ on the connected component $I$ is
either trivial or free. Since the action of $G$ is effective, there is a
connected component $I_{1}$ on which $B_{1,s}$ acts freely. A similar
argument shows that $B_{i,s^{\prime }}$ acts freely on a component $I_{i}$
for each $i\in \mathbb{R}_{\geq 1}$ and any $s^{\prime }\in \mathbb{R}%
\backslash \{0\}.$

Since $B_{i,s^{\prime }}$ commutes with $B_{j,s},$ we have $B_{i,s^{\prime
}}(I_{1})\subset \mathbb{R}\backslash \mathrm{Fix}(\langle B_{j,s}:s\in 
\mathbb{R}\rangle ).$ Moreover, $B_{i,s^{\prime }}(I_{j})\cap I_{j}$ is
either $I_{j}$ or the empty set. Suppose that $I_{i}\cap I_{j}\neq \emptyset
\ $and the right end $b_{i}$ of $I_{i}$ lies in $I_{j}.$ Choose $x\in
I_{i}\cap I_{j}.$ Note that $B_{j,s}([x,b_{i}))\cap \lbrack
x,b_{i})=\emptyset $ for any $s>0.$ This is impossible as $%
B_{j,s/n}(x)\rightarrow x,n\rightarrow \infty .$ Therefore, we have $%
I_{i}\cap I_{j}=I_{i}$ or empty for distinct $i,j\in \mathbb{R}_{\geq 1}$.
Since we have uncountably many $i\in \mathbb{R}_{>0},$ there must be some
distinct $i,j\in \mathbb{R}_{\geq 1}$ such that $I_{i}=I_{j}$. This means
that the subgroup $\mathbb{R}\bigoplus \mathbb{R}$ spanned by the $i,j$%
-components acts freely on $I_{i}.$ H\"{o}lder's theorem (cf. Lemma \ref%
{holder}) gives an injective group homomorphism $\phi :\mathbb{R}\bigoplus 
\mathbb{R}\rightarrow (\mathbb{R},+)$ and a continuous non-decreasing map $%
\varphi :\mathbb{R}\rightarrow \mathbb{R}$ such that%
\[
\varphi (h(x))=\varphi (x)+\phi (h) 
\]%
for any $x\in \mathbb{R}.$ Since $\langle A_{t}:t\in \mathbb{R}_{>0}\rangle
\rightarrow \mathrm{Homeo}(\mathbb{R})$ is a topological embedding, we have
that for any fixed $x\in \mathbb{R}$, $A_{t}(x)$ is continuous with respect
to $t\in \mathbb{R}_{>0}.$ By Corollary \ref{cor}, the restriction map $\phi
|_{\mathbb{R}}$ is $\mathbb{R}$-linear for each direct summand $\mathbb{R}$.
This is a contradiction to the fact that $\phi $ is injective. Therefore,
the group $G\ $cannot act effectively.
\end{proof}

\bigskip

\begin{proof}[Proof of Theorem \protect\ref{th2}]
Suppose that $\mathrm{QI}^{+}(\mathbb{R})$ acts on the real line by an
injective group homomorphism $\phi :\mathrm{QI}^{+}(\mathbb{R})\rightarrow 
\mathrm{Homeo}(\mathbb{R}).$ The group $\mathrm{QI}^{+}(\mathbb{R})$
contains the semi-direct product $\mathbb{R}_{>0}\ltimes (\bigoplus_{i\in 
\mathbb{R}_{\geq 1}}\mathbb{R)}$, by Lemma \ref{lema1}. The subgroup $%
\mathbb{R}_{>0}$ (as the image of the exponential map) is a homomorphic
image of the subgroup $\mathbb{R}<\tilde{\Gamma},$ which is the lifting of $%
\mathrm{SO}(2)/\{\pm I_{2}\}<\mathrm{PSL}_{2}(\mathbb{R})$ to $\mathrm{Homeo}%
(\mathbb{R}).\ $Note that $\tilde{\Gamma}$ is embedded into $\mathrm{QI}^{+}(%
\mathbb{R})$ (see Corollary \ref{corsar} and its proof). By Lemma \ref{mili}%
, any effective action of $\tilde{\Gamma}$ on the real line $\mathbb{R}$ is
a topological diagonal embedding. This means that the action of $\mathbb{R}%
_{>0}$ is a topological diagonal embedding (cf. Corollary \ref{continu}).
Theorem \ref{th3} shows that the action of $\mathbb{R}_{>0}\ltimes
(\bigoplus_{i\in \mathbb{R}_{\geq 1}}\mathbb{R)}$ is not effective.
\end{proof}

\bigskip

\noindent \textbf{Acknowledgements}

The authors would like to thank Li Cai and Xiaolei Wu for their helpful
discussions. This work is supported by NSFC (No. 11971389).

\bigskip

NYU Shanghai, 1555 Century Avenue, Shanghai, 200122, China.

NYU-ECNU Institute of Mathematical Sciences at NYU Shanghai, 3663 Zhongshan
Road North, Shanghai, 200062, China

E-mail: sy55@nyu.edu

\bigskip

Department of Pure Mathematics, Xi'an Jiaotong-Liverpool University, 111 Ren
Ai Road Suzhou Jiangsu, China

E-mail: Yanxin.Zhao19@student.xjtlu.edu.cn


\bigskip

\begin{thebibliography}{9}
\bibitem{cm} L. Chen, K. Mann, \textit{There are no exotic actions of
diffeomorphism groups on 1-manifolds}, arXiv:2003.07452.

\bibitem{dnr} B. Deroin, A. Navas, and C. Rivas, \textit{Groups, orders, and
dynamics}, Submitted, 2016.

\bibitem{gp} M. Gromov, P. Pansu, \textit{Rigidity of lattices: an
introduction}, Geometric Topology: Recent Developments, Lecture Notes in
Math., vol. 1504, Montecatini Terme, 1990, Springer, Berlin (1990), pp.
39-137

\bibitem{mann} K. Mann, \textit{Left-orderable groups that don't act on the
line}, Math. Z. 280 (2015), 905--918.

\bibitem{ma} S. Matsumoto, \textit{Numerical invariants for semiconjugacy of
homeomorphisms of the circle}. Proc. AMS 98 (1986), 163--168 .

\bibitem{mil} E. Militon, \textit{Actions of groups of homeomorphisms on
one-manifold,}. Groups Geom. Dyn., 10 (2016), 45--63.

\bibitem{nav} A. Navas, \textit{On the dynamics of (left) orderable groups},
Annales de l'Institut Fourier, Tome 60 (2010) no. 5, pp. 1685-1740.

\bibitem{na} A. Navas, \textit{Groups of Circle Diffeomorphisms,} Chicago
Lectures in Mathematics (2011).

\bibitem{sar} P. Sankaran, \textit{On homeomorphisms and quasi-isometries of
the real line}, Proc. of the Amer. Math. Soc., 134 (2006), 1875-1880.
\end{thebibliography}
\end{document}